\documentclass[11pt]{article}
\usepackage[english]{babel}
\usepackage{amssymb}
\usepackage{amsthm}
\usepackage{amsmath}
\usepackage{a4wide}
\usepackage{esvect}
\usepackage{hyperref}
\usepackage{verbatim}
\usepackage[dvips]{graphicx}
\usepackage{color}
\usepackage{geometry}
\usepackage{framed}
\newtheorem{thm}{Theorem}%[chapter]
\newtheorem{lem}[thm]{Lemma}

\newtheorem{cor}[thm]{Corollary}

\newtheorem{rem}[thm]{Remark}
\newtheorem{clm}[thm]{Claim}

\newcommand\cA{{\mathcal A}}

\newcommand\cF{{\mathcal F}}
\newcommand\cG{{\mathcal G}}

\newcommand{\ignore}[1]{}

\title{Rounds in a combinatorial search problem}

\begin{document}

\author{
D\'aniel Gerbner\thanks{Research supported by the J\'anos Bolyai Research Fellowship of the Hungarian Academy of Sciences.} 
\and
M\'at\'e Vizer\thanks{Research supported by the National Research, Development and Innovation
Office -- NKFIH under the grant SNN 116095.}}

\date{MTA R\'enyi Institute \\
Hungary H-1053, Budapest, Re\'altanoda utca 13-15.\\
\small \texttt{gerbner@renyi.hu, vizermate@gmail.com}\\
\today}

\maketitle

\begin{abstract}
We consider the following combinatorial search problem: we are given some excellent elements of $[n]$ and we should find at least one, asking questions of the following type: "Is there an excellent element in $A \subset [n]$?". G.O.H. Katona \cite{K2011} proved sharp results for the number of questions needed to ask in the adaptive, non-adaptive and two-round versions of this problem.

We verify a conjecture of Katona by proving that in the $r$-round version we need to ask $rn^{1/r}+O(1)$ queries for fixed $r$ and this is sharp.

We also prove bounds for the queries needed to ask if we want to find at least $d$ excellent elements.

\end{abstract}

\section{Introduction}

In the most basic model of combinatorial search theory Questioner needs to find a special element $x$ of $\{1,2,...,n\}(=:[n])$ by asking minimal number of questions of type "does $x \in F \subset [n]?$". Special elements are usually called defective; in this paper, following \cite{K2011} we call them \textit{excellent}. There are many generalizations of this very basic model, one can find many directions and results in the following survey papers and books: \cite{AW1987,A1988,C2013,DH1999,K1973}.

We call the \textit{complexity} of a specific combinatorial search problem the number of the questions needed to ask by Questioner in the worst case during an optimal strategy. 

For every combinatorial search problem there are at least two main approaches: whether it is \textit{adaptive} or \textit{non-adaptive}. In the adaptive scenario Questioner asks questions depending on the answers for the previously asked questions, however in the non-adaptive version Questioner needs to pose all the questions at the beginning.

A possible intermediate scenario is when there are $r$ rounds for some integer $r \ge 1$ fixed at the beginning and Questioner can pose questions in the $i^{th}$ round ($1 \le i \le r$) depending on the answers for the questions posed in the first $i-1$ rounds. Note that the non-adaptive version is the one-round version, and in the adaptive version there are infinitely many rounds (however it is easy to see that at most $n$ (or some function of $n$) rounds are enough for most of the combinatorial search problems). There are results in the literature, when authors provide a solution for an adaptive search problem that also solves the $r$-round version of that problem for some $r$. However we could only find few examples (see e.g. \cite{W2013}) where the focus of the research is how the complexity changes depending on the number of rounds. Our results fit into this line of research.

The paper is organized as follows: in Subsections 1.1 and 1.2 we state our results and in Section 2 we prove them. Finally we make some remarks and pose some questions.

\subsection{The model}

A question of R. Chambers was answered by G.O.H. Katona \cite{K2011}, who determined a sharp (up to constant terms) result for the complexity of the adaptive, non-adaptive and 2-round versions of the following model.

\vspace{2mm}

$\bullet$ \textbf{Input:} $[n]$ with some (possibly zero) excellent elements.

\vspace{1mm} 

$\bullet$ \textbf{Question:} is there an excellent in $A \subset [n]$?

\vspace{1mm}

$\bullet$ \textbf{Goal:} find an excellent element or state that there is none.

\vspace{2mm}

We denote the $r$-round version of this problem by $P(n,?,1,r)$ and denote by $|P(n,?,1,r)|$ its complexity. We also consider that variant of the previous model (and denote by $P(n,?,d,r)$), when Questioner should find (at least) $d$ excellent elements (or state that there are at most $d-1$), and also use the notation $|P(n,?,d,r)|$ for the complexity of the latter problem.

\subsection{Results}

%\subsection*{$P(n,?,1,r)$}

In the following theorem we verify a conjecture of Katona (\cite{K2011}, Conjecture 1) by determining the complexity of $P(n,?,1,r)$ almost exactly.

\begin{thm}\label{main1}

For any $r,n \ge 1$ we have: $$  rn^{1/r} \ge |P(n,?,1,r)| \ge rn^{1/r}-2r+1.$$

\end{thm}

\vspace{4mm}

%\subsection*{$P(n,?,d,r)$}

We have a larger gap in case we want to find more excellent elements.

\begin{thm}\label{main3}

For any $r \ge 1$ and $n \ge d \ge 2$ we have: $$r\lceil  (d^{r-1}n)^{1/r}\rceil \ge |P(n,?,d,r)| \ge r(dn)^{1/r}-2d-r(d+1)+2.$$

\end{thm}

However note that for two rounds the upper and lower bounds are asymptotically equal as $n$ tends to infinity.

\begin{cor}\label{main2} For any $n \ge d \ge 2$ we have:$$2\lceil  (dn)^{1/2}\rceil \ge |P(n,?,d,2)| \ge 2\lceil  (dn)^{1/2}\rceil-4d-2.$$

\end{cor}

\section{Proofs}

\subsection{Proof of Theorem \ref{main1}}

First we prove the upper bound. To do this we describe an algorithm (given by Katona \cite{K2011}). In the first round Questioner partitions $[n]$ into $\lceil n^{1/r}\rceil$ parts such that their sizes differ by at most one. Then he asks all of these parts except one, $C$ which is one of the smaller parts. Then he picks one of the parts that were answered yes, or if there is no such part, then he picks $C$. In the next round he continues on the picked part recursively, i.e. he partitions it into $\lceil n^{1/r}\rceil$ parts such that their sizes differ by at most one and asks all but one of the smaller parts, and so on. In the last round if all the previous answers were no, he changes the algorithm and asks all the parts instead. It is easy to see that in the last round the parts are of size at most one, thus he finds an excellent element if there is any, and that in each round at most $\lceil n^{1/r}\rceil$ queries were asked.

To prove the lower bound we describe a strategy for Adversary to force Questioner to ask at least $r(n^{1/r}-\frac{r-1}{r}-1)$ questions before reaching his goal. First we introduce the following notation. For $1\le i \le r$ let $\cF_i$ be the family of the queries asked by the Questioner in round $i$ and $k_i:=|\cF_i|$. Let $\cF_i^Y\subset\cF_i$ be the family of queries that are answered yes by Adversary, and let $\cF_i^N\subset \cF_i$ be the family of those queries that are answered no (and so $\cF_i^N = \cF_i \setminus \cF_i^Y$). Let $G_i:=\bigcup(\cup_{j=1}^{i} \cF_j^N)$, the set of those elements that are known to be not excellent after round $i$. Informally we can forget about them, and restrict the underlying set to $[n]\setminus G_i$ after round $i$. Finally let $\cG_i:=\cup_{j=1}^{i} (\cF_j^Y\setminus G_i)$ the set of the queries answered yes during the first $i$ rounds restricted to $[n]\setminus G_i$, $m_i:=\min \{|G|: G\in \cG_i\}$ the cardinality of the smallest set in $\cG_i$ and $n_i:=\lfloor n_{i-1}/(k_i+1)\rfloor\ge n/\Pi_{j=1}^{i} (k_j+1)-i$ (with $n_0=n$, and the latter inequality is an easy consequence of the fact that $k_i \ge 0$). We remark that when we describe how Adversary answers the queries in round $i$, we use only information that Adversary has at that point. For example, $k_1, \dots, k_i$ are known, but $k_{i+1}$ is not known after Questioner poses the questions in round $i$.

When an element appears in a query that is answered no, we know that that element cannot be excellent, thus it does not matter if a later query contains it or not. Hence we can assume without loss of generality that no elements of $G_i$ appear in a member of $\cF_j$ for $j>i$.

The proof of the lower bound for the case of two rounds by Katona essentially consists of two steps. First it is shown that the first round of queries can be answered (by Adversary) in a way that either $m_1$ is large or all the answers are no and $|G_1|$ is relatively small. Afterwards it is shown that in the last round if $\cF_1^Y$ is not empty, at least $m_1-1$ queries are needed, or if $\cF_1^Y$ is empty, then at least $n-|G_1|$ queries are needed. Here in Lemma \ref{altalanos} we extend the first step to more rounds and for sake of completeness we reprove the lemma about the last step (Lemma \ref{utso}). 

Now we show how Adversary should answer during the first $r-1$ rounds.

\begin{lem}\label{altalanos} Adversary can answer $\cF_1,\dots, \cF_{r-1}$ such a way that for all $1 \le t \le r-1$ we have either: 

\vspace{1mm}

$\bullet$ $n_t\le m_t-1$, or 

\vspace{1mm}

$\bullet$ all the answers are no in the first t rounds and $|G_t|\le n-n_t$. 

\end{lem}

\begin{proof} We use induction on $t$ and let us consider round $t$. 

\vspace{3mm}

If $t=1$, then Adversary orders the elements of $\mathcal{F}_1$ in the following way:

\vspace{2mm}

$\bullet$ let $H_1:=F_1$ be one of the smallest sets in $\mathcal{F}_1$,

\vspace{1mm}

$\bullet$ for $2 \le i \le |\mathcal{F}_1|$ let $F_i \in \cF_1 \setminus \{F_1,F_2,...,F_{i-1}\}$ be such that the cardinality of $H_i:=F_i\setminus \cup_{j=i}^{i-1}F_j$ is as small as possible. Note that the sets $H_i$ are disjoint from each other. 

\vspace{1mm}

After this if there is no $i$ with $|H_i| \ge n_1+1$, then Adversary answers no for all questions in $\cF_1$ and we clearly have $|G_1| \le n- n/(k_1+1)\le n-n_1$.

However if there is an $i$ with $|H_i| \ge n_1+1$, then Adversary chooses the smallest such $i$ and answers no to $F_j$ if $j<i$ and yes if $j\ge i$. So each query in $\cF_1^Y$ contains a least $|H_i|\ge n_1+1$ elements not in $\cup_{j=1}^{i-1} H_j(=G_1)$ and we are done with the case $t=1$.

\vspace{3mm}

So assume that $t \ge 2$ and first consider the case when Adversary answered in the previous rounds only no answers. Then - by induction - there are at least $n_{t-1}$ elements we do not know anything about. Adversary restricts the queries to those elements, and do the same as in the first round. That results in either that $m_t-1 \ge n_{t-1}/(k_t+1)\ge n_t$ or only no answers and at least $n_{t-1}/(k_t+1)\ge n_t$ many elements still not appearing in any queries.

\vspace{1mm}

Now we assume that Adversary answered yes at least once in the first $t-1$ rounds, and then every element of $\cG_{t-1}$ has size at least $n_{t-1}$. In this case Adversary essentially do the same as in the first round, so orders the elements of $\cF_t$ the following way (note that every element of $\cF_t$ is in the complement of $G_{t-1}$): 

\vspace{2mm}

$\bullet$ let $H_1:=F_1$ be one of the smallest sets in $\cF_t$, and 

\vspace{1mm}

$\bullet$ for $2 \le i \le |\cF_t|$ let $F_i \in \cF_t \setminus \{F_1,F_2,...,F_{i-1}\}$ is such that the cardinality of $H_i:=F_i\setminus \cup_{j=i}^{i-1}F_j$ is as small as possible. Note that the sets $H_i$ are disjoint from each other. 

\vspace{2mm}

Let us assume first that there is an $i$ with $|H_i|\ge n_t+1$, and consider the smallest such $i$. Then Adversary answers no to $F_j$ if $j<i$ and yes if $j\ge i$. Then each query in $\cF_t^Y$ contains a least $|H_i|\ge n_t+1$ elements not in $\cup_{j=1}^{i-1} H_j$. This means those members of $\cG_t$ that correspond to queries in round $t$ have indeed size at least $n_t+1$. The other members - by induction - had size at least $n_{t-1}+1$ before the round, and at most $|\cup_{j=1}^{i} H_j|\le k_t n_t$ elements were moved to $G_t$, thus deleted from them in the current ($t^{th}$) round. Then at least $n_t+1$ remains in each.
 
If there is no such $i$, then Adversary answers no to every query. As earlier there was a yes answer, we still have to show that $n_t\le m_t-1$, but this time we do not have to deal with the new queries. For the earlier queries the same argument works: at most $|\cup_{j=1}^{i} H_j|\le k_t n_t$ elements were deleted from each set in $\cG_{t-1}$ and we are done with the proof of Lemma \ref{altalanos}.

\end{proof}

The following lemma, which deals with the last round is essentially the generalization of Lemma 3.6 in \cite{K2011}, however we provide a proof somewhat more compact than the one in \cite{K2011}, since we want to generalize the argument during
the proof of Theorem \ref{main3}.

\begin{lem}\label{utso} For $r \ge 2$ to be able to find an excellent element in the $r^{th}$ round Questioner needs at least $m_{r-1}-1$ queries if there is at least one yes answer in the first $r-1$ rounds and at least $n- |G_{r-1}|$ queries are needed if all the answers were no in the first $r-1$ rounds.

\end{lem}

\begin{rem} Before starting the proof, note that in Lemma \ref{utso} there is no indication about Adversary's strategy during the first $r-1$ rounds. So the statement of the lemma is true for any strategy.

\end{rem}

\begin{proof}[Proof of Lemma \ref{utso}]

We prove by induction on $n- |G_{r-1}|+m_{r-1}$.

\vspace{2mm}

Note that if $m_{r-1}=0$ (so there were no 'yes' answers during the first $r-1$ rounds), then we are done by the result of Katona (\cite{K2011}, Theorem 2.5) on the non-adaptive version of this problem.

If $m_{r-1}=1$, then we are also done, since there is a one-element query with containing exactly one excellent element.

Using that $n- |G_{r-1}| \ge m_{r-1}$, we are done with the cases $n- |G_{r-1}|+m_{r-1}=1,2,3$.

\vspace{3mm}

So suppose $n- |G_{r-1}|+m_{r-1} \ge 4$ and $m_{r-1} \ge 2$. We claim the following:

\begin{clm}\label{redik}
Questioner should ask a one-element set in the $r^{th}$ round.

\end{clm}

\begin{proof}[Proof of Claim \ref{redik}]
We prove by contradiction. Suppose all queries are of size at least two and all the answers are yes in the $r^{th}$ round, and Questioner can point an excellent element. Let us assume all the elements in $[n] \setminus G_{r-1}$ are excellent, except the one Questioner pointed. This is compatible with the previous answers using that $m_{r-1} \ge 2$, and also with the new answers, a contradiction.

\end{proof}

To continue the proof of Lemma \ref{utso} we can suppose that Questioner asks a one element query  ($x \in [n] \setminus G_{r-1}$) in the $r^{th}$ round. But then Adversary can say no to $\{x\}$ first (this is compatible with the answers in the first $r-1$ rounds, since $m_{r-1} \ge 2$) and consider it as if it were asked during the first $r-1$ rounds and delete $x$ from the remaining queries asked in the $r^{th}$ round. Note that in this new scenario $m_{r-1}$ can decrease by at most 1. As $|n \setminus (G_{r-1} \cup x)| < n - |G_{r-1}|$, since $x \not \in G_{r-1}$ by induction we know that Questioner should ask at least $m_{r-1}-1$ queries and we are done with the proof of Lemma \ref{utso}.

\end{proof}

So Lemma \ref{utso} and Lemma \ref {altalanos} shows that we have that $$k_1+...+k_{r-1}+\frac{n}{(k_1+1)...(k_{r-1}+1)}-r$$ is a lower bound on $|P(n,?,1,r)|$. Using some reorganization and the inequality of arithmetic and geometric means we have:$$k_1+...+k_{r-1}+\frac{n}{(k_1+1)...(k_{r-1}+1)} \ge r(n^{1/r}-\frac{r-1}{r}-1),$$

and we are done with the lower bound and with the proof of Theorem \ref{main1}.

\qed

\subsection{More excellent elements, proof of Theorem \ref{main3}}

The upper bound is given by a straightforward extension of the algorithm constructed in the proof of Theorem \ref{main1}. For simplicity of the description here Questioner will ask every part in a partition. It is not always necessary, but it adds at most $d$ additional queries in each round. (Thus the following algorithm could be easily improved a little.)

In the first round Questioner partitions $[n]$ into $k:=\lceil  (d^{r-1}n)^{1/r}\rceil$ parts such that their sizes differ by at most one. He asks all of them. Then he picks $d$ of the parts that were answered yes, or if there are less than $d$ such parts, then he picks all of them. He continues on each of these parts simultaneously, i.e. he partitions each of them into $\lceil (n/k)^{1/(r-1)}\rceil$ parts such that their sizes differ by at most one, and asks each of those parts. Then he picks $d$ or all of the parts that were answered yes, and so on. At the end of the  first round the size of a part is at most $n$ divided by $k$ and then in further rounds divided by $\lceil (n/k)^{1/(r-1)}\rceil$, altogether $r-1$ times, thus it is at most $1$ at the end of the $r^{th}$ round. If at any round there were $d$ different yes answers, he finds an excellent element in each of those sets, if not, then at the last round he asks every singleton that can still be excellent, thus finds all of the excellent elements.

In the first round there are $k$ queries, in any of the later rounds there are at most $d\lceil (n/k)^{1/(r-1)}\rceil$ queries, which give altogether at most $$k+(r-1)d\lceil (n/k)^{1/(r-1)}\rceil \le r\lceil  (d^{r-1}n)^{1/r}\rceil$$ queries, that proves the upper bound.

\vspace{3mm}

For the lower bound we prove the generalizations of Lemma \ref{utso} and Lemma \ref{altalanos}, but we need to modify Adversary's strategy for round $r-1$. We use the notation introduced in the proof of Theorem \ref{main1}. Additionally, for a family $\cF$ of subsets of $[n]$ and $1 \le i \le d$ let

$$m(i,\cF):=\min\{|\cup^i_{j=1} A_j|: A_j\in \cF \textrm{ for } 1 \le j \le i\},$$

\noindent
and for $1 \le i \le d$ and $1 \le t \le r-1$ let $$m_t(i):=m(i,\cG_t).$$ %and let $$n_t:=\lfloor\frac{nd}{\Pi_{j=1}^{t} (k_j+1)}\rfloor (IDE EZ KELL?)$$ during this proof. 
We also assume that the members of $\cF_t$ are in the complement of $G_{t-1}$.

\begin{lem}\label{daltalanos} Adversary can answer $\cF_1,\dots, \cF_{r-2}$ such a way that for all $1 \le t \le r-2$ either 

\vspace{2mm}

$\bullet$ all the answers are no in the first $t$ rounds and $|G_t|\le n-n_t$, or 

\vspace{1mm}

$\bullet$ for $1 \le i \le d$ we have $m_t(i)\ge in_t$. 

\end{lem}

\begin{proof} We use induction on $t$ and let us consider round $t$ (with $G_0=\emptyset$).

\vspace{2mm}

We say that a family $\cA\subset \cF_t$ is \textit{good} if there are no $i \le d$ and $A_1,...,A_i \in \cF_t \setminus \cA$ with $$|\cup_{j=1}^{i} A_j \setminus \cup \cA |< i n_t.$$

%Adversary chooses $\cF'_t \subset \cF_t$ with the smallest cardinality such that there are no $A_1,...,A_i \in \cF_t \setminus \cF'_t$ with $$|\cup_{j=1}^{i} A_j \setminus \cup \cF'_t |< i n_t$$ for some $i \le d$. Szerintem: a family is good ha ezt tudja.

%Note that it is well-defined, as $\cF_t$ is a family satisfying the condition .
To prove Lemma \ref{daltalanos} first we need the following claim:

\begin{clm}\label{bound} 
%Defining $\cF'_t$ as above we have $$|\cup \cF'_t| \le |\cF'_t|n_t.$$
There is a good family $\cA$ with $$|\cup \cA| \le |\cA|n_t.$$

\end{clm}

\begin{proof}
%Az en valtozatom. Ez a $\cF'_t$ kicsit mas, nem min meretu.

One can build a good family the following greedy way. Starting with the empty family, at the first step we pick  $A^1_1,...,A^1_{i_1} \in \cF_t $ with $|\cup_{j=1}^{i_1} A^1_j |<i_1 n_t$ for some $i_1 \le d$ if such sets exist. Then let $\cA_1:=\{A^1_1, \dots, A^1_{i_1}\}$. For $s \ge 2$ in the $s^{th}$ step we pick $A^s_1,...,A^s_{i_s} \in \cF_t \setminus \cA_{s-1}$ with $|\cup_{j=1}^{i_s} A^s_j \setminus \cup \cA_{s-1} |< i_s n_t$ for some $i_s \le d$. If there is such a set, let $\cA_s:=\cA_{s-1}\cup \{A^s_1, \dots, A^s_{i_s}\}$. If there are no such sets,  let $\cA := \cA_{s-1}$, that is obviously a good family. As $\cF_t$ is a finite set, in finitely many steps we arrive to the later case.

We will show that for every $s\ge 0$ we have $|\cup \cA_s|\le |\cA_s|n_t$. We prove it by induction on $s$. It obviously holds for $s=0$. Let us assume it holds for $s-1$, and consider the last step. We have $|\cup \cA_{s-1}|\le |\cA_{s-1}|n_t$ and $\cA_s=\cA_{s-1}\cup \{A^s_1, \dots, A^s_{i_s}\}$, where $|\cup_{j=1}^{i_s} A^s_j \setminus \cup \cA_{s-1} |< i_s n_t$. These imply the statement.

\end{proof}

\noindent We remark that the above claim holds with strict inequality except if the only good family is the empty family.

\vspace{3mm}

In the $t^{th}$ round Adversary answers no to all queries in $\cA$, where $\cA$ is a good family satisfying Claim \ref{bound} (i.e. $\cF^N_t:=\cA$) and yes for all queries in $\cF_t \setminus \cA$. Note that in case $d=1$ it is a more general and compact form of writing down what Adversary does in (the proof of) Lemma \ref{altalanos},

\vspace{2mm}

Now we continue the proof of Lemma \ref{daltalanos}. We have 2 main cases:

\vspace{3mm}

$\textbf{Case 1:}$ Adversary gave no answers for all the queries in the first $t-1$ rounds.

\vspace{1mm}

\hspace{5mm} $\textbf{Case 1/a:}$ If $\cF^N_t=\cF_t$. 

 By induction we know that $|G_{t-1}| \le n-n_{t-1}$ and by Claim \ref{bound} we have $|\cup \cF_t| \le |\cF_t|n_t$, which proves $|G_t| \le n-n_t$ using that $n_{t-1} \ge (k_t+1)n_t$.

\vspace{2mm}

\hspace{5mm} $\textbf{Case 1/b:}$ If $\cF^N_t \neq \cF_t$. 

By the goodness of $\cF^N_t$ we are done.

\vspace{3mm}

$\textbf{Case 2:}$ Adversary gave at least one yes answer during the first $t-1$ rounds. We want to prove that for any $i\le d$ members $A_1,\dots, A_i \in \cG_t$ we have $|\cup_{j=1}^{i} A_j|\ge i n_t$. Note that by definition we have $n_{t-1}\ge (k_t+1)n_t$.

\vspace{1mm}

\hspace{5mm} $\textbf{Case 2/a:}$ $A_1,\dots, A_i \in \cF_t \setminus \cF^N_t.$

By the goodness of $\cF^N_t$ we have that $|\cup_{j=1}^i A_j|\ge in_t$.

\vspace{2mm}

\hspace{5mm} $\textbf{Case 2/b:}$ If there is $1 \le e \le i$ with $A_1,...,A_e \in \cG_{t-1}$ and $A_{e+1},...,A_i \in \cF_t \setminus \cF^N_t$.

In this case we know by the induction on $t$ that $$|A_1\cup...\cup A_e| \ge en_{t-1}\ge e(k_t+1)n_t,$$
and by Claim \ref{bound} that $$ |\cup \cF^N_t| \le (k_t-(i-e))n_t,$$
as $|\cF_t \setminus \cF^N_t| \ge i-e$. So we have $$|\cup_{j=1}^{i} A_j| \ge |A_1 \cup ... \cup A_e \setminus \cup \cF^N_t| \ge e(k_t+1)n_t-(k_t-(i-e))n_t \ge in_t,$$

and we are done with the proof of Lemma \ref{daltalanos}.

\end{proof}

Now we deal with the penultimate round. Let $n'_{r-1}:=\lfloor n_{r-2}/(k_{r-1}+d)\rfloor$.

\begin{lem}\label{daltalanos2} Adversary can answer $ \cF_{r-1}$ such a way that either 

\vspace{2mm}

$\bullet$ all the answers are no in the first $r-1$ rounds and $|G_{r-1}|\le n-dn'_{r-1}$, or 

\vspace{1mm}

$\bullet$ for $1 \le i \le d$ we have $m_{r-1}(i)\ge in'_{r-1}$. 

\end{lem}

We state this lemma separately in order to emphasize that Adversary modifies the strategy for the penultimate round. However, the proof is essentially the same, thus we only give a sketch here.

\begin{proof}[Sketch of the proof of Lemma \ref{daltalanos2}] Here we call a family good if there are no $i \le d$ and $A_1,...,A_i \in \cF_t \setminus \cA$ with $|(\cup_{j=1}^{i} A_j) \setminus \cup \cA |< i n'_{r-1}.$ Similarly to Claim \ref{bound}, there is a good family $\cA$ with $|\cup \cA|\le |\cA|n'_{r-1}$. Adversary answers no to queries in $\cA$ and yes to other queries. By the goodness of $\cA$ we have $m_{r-1}(i)\ge in'_{r-1}$. On the other hand if $\cA=\cF_{r-1}$, then $|\cup \cA|\le |\cA|n'_{r-1}=|\cF_{r-1}|n'_{r-1}$, and we have $|G_{r-1}|=|G_{r-2}|+|\cup \cA|$. These together imply  $|G_{r-1}|\le n-dn'_{r-1}$.

\end{proof}

After $r-1$ rounds, there are two very different possibilities. Either the smallest set that intersects every member of $\cG_{r-1}$ has cardinality at least $d$, in which case there are at least $d$ excellent elements in $[n]$, or it has cardinality less than $d$, in which case it is possible that there are less than $d$ excellent elements in $[n]$.

\begin{lem}\label{dutso} 
%Suppose $m^d_{r-1} \ge d$ or there were no 'yes' answers in the first $r-1$ rounds. In order to be able to find $d$ defective elements or to show that there are less than $d$ defective elements, 

Let $A$ be the smallest set that intersects every member of $\cG_{r-1}$. In the $r^{th}$ round Questioner needs to ask 

\vspace{2mm}

$\bullet$ at least $m_{r-1}(d)-d$ queries if $|A|\ge d$, and

\vspace{1mm}

$\bullet$ at least $n- |G_{r-1}|-d+1$ queries if $|A|<d$.

\end{lem}

\begin{rem} We also remark here that there is no indication about Adversary's strategy during the first $r-1$ rounds in Lemma \ref{dutso}. So the statement of the lemma is true for any strategy.

\end{rem}

\begin{proof} We use induction on $n-|G_{r-1}|+m_{r-1}(d).$ 

\vspace{2mm}

%If all the answers were no in the first $r-1$ rounds, then by the result of Katona (\cite{K2011}, Theorem 2.5) - as in the proof of Lemma \ref{utso} - at least $n-|G_{r-1}|$ queries are needed even to find one defective element and we are done. 

\vspace{2mm}

If $|A|\ge d$ then we can suppose that $m_{r-1}(d) \ge d+1$ and $n-|G_{r-1}|\ge d$ (otherwise Lemma \ref{dutso} obviously holds). Assuming these we state the following:

\begin{clm}\label{dredik}

Questioner has to ask a one-element set in the $r^{th}$ round. Moreover, if $|A|<d$, Questioner has to ask a one-element set disjoint from $A$.

\end{clm}

\begin{proof}[Proof of Claim \ref{dredik}]

We prove by contradiction, so suppose otherwise, Questioner does not ask any one element set. Let us assume first $|A|\ge d$. Then Adversary answers 'yes' to all queries. Suppose that Questioner could point $d$ excellent elements. Then, by $m_{r-1}(d) \ge d+1$ there is a pointed element $x \in [n]$ such that all members of $\cG_{r-1}$ that contain $x$ have cardinality at least two. So it is possible that $x$ is not excellent and all the elements in $[n] \setminus (G_{r-1} \cup \{x\})$ are excellent, as this is compatible with the answers given during the first $r-1$ rounds and even in the $r^{th}$ round, and in this case Questioner should not point $x$, which is a contradiction.

Assume now that $|A|<d$. Adversary again answers 'yes' to all queries. Suppose that Questioner could point $d$ excellent elements. Then there is a pointed element $x\not\in A$. If $\{x\}$ is not asked then it is possible that $x$ is not excellent and all the elements in $[n] \setminus (G_{r-1} \cup \{x\})$ are excellent, since the elements of $A$ are excellent, thus the answer given during the first $r-1$ rounds are compatible, while the answers in the last round are compatible (unless $\{x\}$ was asked). Finally assume that Questioner claims that there are less than $d$ excellent elements. Obviously it is possible that all the elements of $[n] \setminus G_{r-1}$ are excellent, and there are at least $d$ of them, a contradiction.

\end{proof}

Let us continue the proof of Lemma \ref{dutso} and assume first that $|A|\ge d$. By Claim \ref{dredik} Questioner asks a one-element set $\{x\}$ in the $r^{th}$ round. We can suppose that all members of $\cG_{r-1}$ that contain $x$ have cardinality at least two, since otherwise Questioner should not ask $\{x\}$ in the $r^{th}$ round. Adversary will answer no to $\{x\}$, and puts it into $G_{r-1}$ and delete from the queries in $\cG_{r-1}$ and from the queries that were asked in the $r^{th}$ round and contained $x$. Let us denote by $\cG_{r-1}'$ and $\cF'_r$ these families. This operation is compatible with all the previous answers (since it can be the case that all the elements in $[n] \setminus (G_{r-1} \cup \{x\})$ are excellent), and the queries asked by Questioner should solve this problem (meaning that Questioner can point $d$ excellent or can state that there are less than $d$).  Then for all $1 \le i \le d$ we have $$m_{r-1}(i) \ge m(i,\cG'_{r-1}) \ge m_{r-1}(i)-1.$$

However in this situation we have that $n-|G_{r-1} \cup \{x\}|+m(d,\cG'_{r-1}) < n-|G_{r-1}|+m_{r-1}(d)$, so by induction - and by $m(d,\cG'_{r-1}) \ge d$ Questioner should ask at least $m(d,\cG'_{r-1})-d$ questions. As $m(d,\cG'_{r-1}) \ge m_{r-1}(d)-1$ we are done.

Let us assume now $|A|<d$. By Claim \ref{dredik} Questioner asks a one-element set $\{x\}$ that is disjoint from $A$ in the $r^{th}$ round. In this case we know that all members of $\cG_{r-1}$ that contain $x$ have cardinality at least two, as they also intersect $A$. Then like in the previous case Adversary answers no to $\{x\}$ and considers it as if it was asked in an earlier round. $\cG_{r-1}'$ and $\cF'_r$ are defined similarly. $A$ intersects every member of $\cG_{r-1}'$, thus at least $n-|G_{r-1}\cup \{x\}|-d+1$ further questions are needed, thus we are done with the proof of Lemma \ref{dutso}.

\end{proof}

By Lemma \ref{daltalanos} and Lemma \ref{dutso} we have that 
$$|P(n,?,d,r)| \ge k_1+\dots+k_{r-1}+dn'_{r-1}-d$$ $$\ge k_1+\dots+k_{r-1}+\frac{nd}{(k_1+1)\dots(k_{r-2}+1)(k_{r-1}+d)}-d-rd$$
and using again some reorganization and the inequality of the geometric and arithmetic means we are done with the proof of Theorem \ref{main3}.

\qed

\section*{Questions, Remarks}

To finish this article we pose a couple questions:

\vspace{2mm}

$\bullet$ The first one is about the statement of Theorem \ref{main3}. It would be interesting to find the same multiplicative factor of $n^{1/r}$ in a lower and an upper bound thus determine the asymptotic of $|P(n,?,d,r)|$.

\vspace{2mm}

Note that in the case $r=\lceil \log n \rceil$ (so basically in the adaptive case) Theorem \ref{main1} does not give back the adaptive result of Katona(\cite{K2011}).

\vspace{1mm}

$\bullet$ It would be interesting to determine the asymptotics of $|P(n,?,d,r)|$, when $r$ or $d$ is a function of $n$ that goes to infinity with $n$.

\vspace{2mm}

$\bullet$ In this paper we assumed nothing in advance about the number of excellent elements. One could consider different models where we know that there are exactly, at most, or at least $e$ excellent elements in $[n]$.

\section*{Acknowledgement}

We would like to thank all participants of the Combinatorial Search Seminar at the Alfr\'ed R\'enyi Institute of Mathematics for fruitful discussions.

\end{document}